\documentclass[12pt]{amsart}
\usepackage{amssymb,amsmath,amsthm,slashed}
\newtheorem{theorem}{Theorem}
\newtheorem{lemma}[theorem]{Lemma}

\newtheorem*{conjecture}{\bf Conjecture}
\newtheorem{proposition}[theorem]{Proposition}

\theoremstyle{remark}

\newtheorem*{remark}{Remark}

\numberwithin{theorem}{section} \numberwithin{equation}{section}

\newcommand{\SU}{\mathrm{SU}}

\newcommand{\C}{\mathbb{C}}

\newcommand{\Z}{\mathbb{Z}}

\newcommand{\sgn}{\operatorname{sgn}}

\newcommand{\re}{\textnormal{Re}}
\newcommand{\im}{\textnormal{Im}}
\def\H{\mathbb{H}}

\begin{document}

\title[$\mathrm{SU}(2)$-Donaldson invariants of the complex projective plane]
{$\mathrm{SU}(2)$-Donaldson invariants of the complex projective plane}

\author{Michael Griffin, Andreas Malmendier and Ken Ono}

\address{Department of Mathematics and Computer Science, Emory University,
Atlanta, Georgia 30322} \email{mjgrif3@emory.edu}
\email{ono@mathcs.emory.edu}

\address{Department of Mathematics, Colby College,
Waterville, Maine 04901}
\email{andreas.malmendier@colby.edu}

\thanks{The first and third authors thank
the generous support of the National Science Foundation. The third author
thanks the support of the Asa Griggs Candler Fund.}

\begin{abstract}
There are two families of Donaldson invariants for the complex projective plane,
corresponding to the $\mathrm{SU}(2)$-gauge theory and the $\mathrm{SO}(3)$-gauge theory
with non-trivial Stiefel-Whitney class.
In 1997 Moore and Witten \cite{MooreWitten} conjectured that the regularized $u$-plane integral
on $\mathbb{C}\mathrm{P}^2$
gives the generating functions for these invariants. In earlier work \cite{M_Ono},
the second two authors proved the conjecture for the $\mathrm{SO}(3)$-gauge theory.
Here we complete the proof of the conjecture by confirming the claim
 for the $\mathrm{SU}(2)$-gauge theory. As a consequence, we find that the
 $\SU(2)$ Donaldson invariants for $\mathbb{C}\mathrm{P}^2$ are explicit linear combinations of the Hurwitz class numbers which arise in the theory of imaginary quadratic
 fields and orders.
\end{abstract}

\maketitle

\section{Introduction and Statement of Results}

Donaldson invariants of smooth simply connected four-dimensional manifolds \cite{DonaldsonKronheimer} are diffeomorphism class invariants which play a central
role in differential topology and mathematical physics.
There are two families of Donaldson invariants, corresponding to the $\mathrm{SU}(2)$-gauge
theory and the $\mathrm{SO}(3)$-gauge theory with non-trivial Stiefel-Whitney class. In each case, the invariants are graded homogeneous
polynomials on the homology $H_0(\mathbb{C}\mathrm{P}^2) \oplus H_2(\mathbb{C}\mathrm{P}^2)$, where $H_i(\mathbb{C}\mathrm{P}^2)$ is considered to
have degree $(4-i)/2$, defined using the fundamental homology classes of the corresponding moduli spaces of anti-selfdual instantons arising in gauge
theory. These invariants are typically very difficult to calculate.


Here we consider the simplest manifold to which Donaldson's definition applies, the complex projective plane
$\mathbb{C}\mathrm{P}^2$ with the Fubini-Study metric.
In earlier work, G\"ottsche and Zagier \cite{GoettscheZagier} gave a formula for the Donaldson
invariants of rational surfaces in terms of theta functions of indefinite lattices.
As an application, G\"ottsche \cite{Goettsche} derived
closed expressions for the two families of the Donaldson invariants of $\mathbb{C}\mathrm{P}^2$
assuming the truth of the Kotschick-Morgan conjecture.  Recently, G\"ottsche,
Nakajima, Hiraku,  and Yoshioka
\cite{GoettscheNakajimaYoshioka} have unconditionally
proved these formulas.

This paper concerns deep conjectures relating these formulas to
constructions in theoretical physics.
From the viewpoint of theoretical physics \cite{Witten1}, these
two families of Donaldson invariants and the related Seiberg-Witten invariants
are the correlation functions of a supersymmetric topological
gauge theory with gauge group $\mathrm{SU}(2)$ and $\mathrm{SO}(3)$.
Witten \cite{Witten2} argued that one should be able
to compute these correlation
functions in a so called {\it low
energy effective field theory}. This theory has the advantage of being an
{\em abelian} $\mathcal{N}=2$ supersymmetric topological gauge theory, and
the data required to define the theory
only involves line bundles of even (resp. odd) first Chern class on
$\mathbb{C}\mathrm{P}^2$ if the gauge group is $\mathrm{SU}(2)$ (resp. $\mathrm{SO}(3)$).
The vacua of the low energy effective field theory are parametrized by the $u$-plane which Seiberg and Witten \cite{SeibergWitten1} describe in terms of the classical
modular curve $\H/\Gamma_0(4)$,
together with a meromorphic one-form. Finally, Moore and Witten \cite{MooreWitten} obtained the correlation functions as regularized integrals over the
$u$-plane, where the integrands are modular functions which are determined by the gauge group.
These regularized $u$-plane integrals define a way of extracting
certain contributions for each boundary
component near the cusps at $\tau=0,2,\infty$ of the modular curve, and Moore and Witten observed \cite{MooreWitten} that the
cuspidal
contributions
at $\tau=0,2$ vanish trivially. This vanishing
corresponds to the mathematical statement that
the Seiberg-Witten invariants on $\mathbb{C}\mathrm{P}^2$
vanish due to the presence of a
Fubini-Study metric of positive scalar curvature \cite{Witten3}.

Concerning the contribution from the cusp $\tau=\infty$,
Moore and Witten made the following deep
conjecture which relates the $u$-plane integral to Donaldson invariants.

\begin{conjecture}[Moore and Witten \cite{MooreWitten}]
The contribution at $\tau=\infty$ to the regularized $u$-plane integral is the generating function for the
Donaldson invariants of $\mathbb{C}\mathrm{P}^2$.
\end{conjecture}

As evidence for this conjecture,
in the case of the gauge group $\mathrm{SU}(2)$, Moore and Witten \cite{MooreWitten} computed the first 40 invariants and found them to be in agreement
with the results of Ellingsrud and G\"ottsche \cite{EllingsrudGoettsche}.
In recent work, the second and third authors proved this conjecture
in the $\mathrm{SO}(3)$ case.
Here we complete the proof of the conjecture by confirming the claim for the gauge group $\mathrm{SU}(2)$.
We prove the following theorem.

\begin{theorem}\label{MainTheorem}
The conjecture of Moore and Witten in the case of the $\mathrm{SU}(2)$-gauge theory on $\mathbb{C}\mathrm{P}^2$ is true.
\end{theorem}

This paper is organized as follows. In Section~\ref{Section2} we recall the work of
G\"ottsche and his collaborators, and we recall the work of Moore and Witten on the $u$-plane
integral. We conclude this section with a result which gives a criterion
(see Theorem~\ref{Criterion}) for
proving Theorem~\ref{MainTheorem}. In Section~\ref{Section3} we prove some
$q$-series identities, and we then combine the theory of modular forms with
the combinatorial properties of certain differential operators to prove
Theorem~\ref{Criterion}, thereby proving Theorem~\ref{MainTheorem}.


\section{Some relevant classical functions}
Here we fix notation concerning theta functions, and we recall a few standard facts about
Dedekind's eta-function and the {\it nearly} modular Eisenstein series $E_2(\tau)$.
We use the following normalization for the Jacobi theta function
\begin{equation}
 \vartheta_{ab}(v|\tau) = \sum_{n\in\mathbb{Z}} q^{\frac{(2n+a)^2}{8}} \; e^{\pi i \,(2n+a)(v+\frac{b}{2})},
\end{equation}
where $a,b \in \lbrace 0,1\rbrace$, $v\in\mathbb{C}$, $q=\exp(2\pi i \tau)$,
$\tau = x+iy \in \mathbb{H}$, and $\mathbb{H}$ is the complex upper half-plane.
The relation to the standard Jacobi theta functions is summarized in the following table:

\smallskip
\begin{equation}
\label{JacobiTheta}
 \begin{array}{|l|l|l|}
 \hline
 \ & \ & \\
  \vartheta_1(v|\tau) = \vartheta_{11}(v|\tau)
& \vartheta_1(0|\tau) = 0
& \vartheta_1'(0|\tau) = - 2\pi \eta^3(\tau) \\ [1ex]
\hline && \\[-2ex]
  \vartheta_2(v|\tau) = \vartheta_{10}(v|\tau)
& \vartheta_2(0|\tau) = \sum_{n\in\mathbb{Z}} q^{\frac{(2n+1)^2}{8}}
& \vartheta_2'(0|\tau) = 0 \\ [1ex]
\hline && \\[-2ex]
  \vartheta_3(v|\tau) = \vartheta_{00}(v|\tau)
& \vartheta_3(0|\tau) = \sum_{n\in\mathbb{Z}} q^{\frac{n^2}{2}}
& \vartheta_3'(0|\tau) = 0 \\ [1ex]
\hline && \\[-2ex]
  \vartheta_4(v|\tau) = \vartheta_{01}(v|\tau)
& \vartheta_4(0|\tau) = \sum_{n\in\mathbb{Z}} (-1)^n \, q^{\frac{n^2}{2}}
& \vartheta_4'(0|\tau) = 0\\
\hline
\end{array}
\end{equation}

\smallskip
Here $\eta(\tau)$ is the Dedekind eta-function with
\begin{equation}
 \eta^3(\tau) = \sum_{n=0}^\infty (-1)^n \; (2n+1) \; q^{\frac{(2n+1)^2}{8}} \;.
\end{equation}
We will also use the notation $\vartheta_j(\tau)=\vartheta_j(0|\tau)$ for $j=2,3,4$, and
\begin{equation}\label{Thetas}
\begin{split}
 \vartheta_2(\tau) = 2 \, \Theta_2\left( \frac{\tau}{8}\right) \,, \quad
 \vartheta_3(\tau) = \Theta_3\left( \frac{\tau}{8}\right)  \,, \quad
 \vartheta_4(\tau) = \Theta_4\left( \frac{\tau}{8}\right) \;.
\end{split}
\end{equation}
Also, we have that $E_2(\tau)$ is the normalized {\it nearly modular} weight 2 Eisenstein series
\begin{equation}
E_2(\tau):=1-24\sum_{n=1}^{\infty}\sum_{d\mid n}d \; q^n \;.
\end{equation}

\section{$\mathrm{SU}(2)$-Donaldson invariants on $\mathbb{C}\mathrm{P}^2$}\label{Section2}

Here we recall a closed formula expression for these Donaldson invariants which is due to G\"ottsche
and his collaborators \cite{Goettsche,GoettscheNakajimaYoshioka}, and we recall the
conjecture of Moore and Witten in this case. We then conclude this section
with Theorem~\ref{Criterion} which we shall use to prove Theorem~\ref{MainTheorem}.

The Donaldson invariants of a smooth, compact, oriented, simply connected Riemannian four-manifold
$(X,g)$ without boundary are defined by using intersection theory on the moduli space of
anti-self-dual instantons for the gauge groups $\mathrm{SU}(2)$ and $\mathrm{SO}(3)$ \cite{Goettsche2}.
Given a homology orientation some cohomology classes on the instanton moduli space can be associated to homology classes of $X$ through the slant
product and then evaluated on a fundamental class. We define $$\mathbf{A}(X):={\rm Sym}(H_0(X,\mathbb{Z}) \oplus H_2(X,\mathbb{Z})),$$ and we
regard the Donaldson invariants as the functional
\begin{equation}
  \mathcal{D}_{w_2(E)}^{X,g}: \mathbf{A}(X) \rightarrow \mathbb{Q} \;,
\end{equation}
where $w_2(E) \in H^2(X, \mathbb{Z}_2)$ is the second Stiefel-Whitney class
of the gauge bundles which are considered.
Since $X$ is simply connected, there is an integer class
$2\lambda_0 \in H^2(\mathbb{C}\mathrm{P}^2,\mathbb{Z})$ that is not divisible by two and
whose mod-two reduction is $w_2(E)$.
Let $\{\mathrm{s}_i\}_{i=1,\ldots, b_2}$ be a basis of the two-cycles of $X$. We introduce
the formal sum $S=\sum_{i=1}^{b_2} \kappa^i \,\mathrm{s}_i$, where  $\kappa^i$ are complex
numbers. The generator of the zero-class of $X$ will be denoted by $\mathrm{x} \in H_0(X, \mathbb{Z})$.
The Donaldson-Witten generating function is
\begin{equation}
\label{donwi}
 Z_{\mathrm{DW}}(p, \kappa)=\mathcal{D}^{X,g}_{w_2(E)}(e^{p \, \mathrm{x} + S}) \;,
\end{equation}
so that the Donaldson invariants are read off from the expansion
of (\ref{donwi}) as the coefficients of powers of $p$ and $\kappa=(\kappa^1, \dots, \kappa^{b_2})$.


In the case of the complex projective plane $\mathbb{C}\mathrm{P}^2$, we have $b_2=b_2^+=1$.
The Fubini-Study metric $g$ on $\mathbb{C}\mathrm{P}^2$ is K\"ahler with the
K\"ahler form $K= \frac{i}{2} g_{a\bar{b}} \, dz^a \wedge dz^{\bar{b}}$.
We denote the first Chern class of the dual of the hyperplane bundle over
$\mathbb{C}\mathrm{P}^2$ by $\operatorname{H} = K/\pi$,  so that
$\int_{\mathbb{C}\mathrm{P}^2} \operatorname{H}^2 =1$, $c_1(\mathbb{C}\mathrm{P}^2)=3 \operatorname{H}$, and $p_1(\mathbb{C}\mathrm{P}^2)=3 \operatorname{H}^2$.
The Poincar\'e dual $\operatorname{h}$ of $\operatorname{H}$ is a generator of the rank-one homology group
$H_2(\mathbb{C}\mathrm{P}^2,\mathbb{Z})$. The $\mathrm{SO}(3)$-bundles on four-dimensional manifolds are classified by the second Stiefel-Whitney class
$w_2(E) \in H^2(\mathbb{C}\mathrm{P}^2,\mathbb{Z}_2)$, and the first Pontrjagin class $p_1(E) \in H^4(\mathbb{C}\mathrm{P}^2,\mathbb{Z})$, such that
\begin{equation*}
 p_1(E)[\mathbb{C}\mathrm{P}^2] \equiv w_2^2(E)[\mathbb{C}\mathrm{P}^2]  \pmod{4}\;.
\end{equation*}
Since $\mathbb{C}\mathrm{P}^2$ is simply connected, there is an integer class
$2 \lambda_0 \in H^2(\mathbb{C}\mathrm{P}^2,\mathbb{Z})$ whose mod-two reduction is $w_2(E)$.
Then, there is a smooth complex two-dimensional vector bundle $\xi \to \mathbb{C}\mathrm{P}^2$ with the Chern classes
$c_1(\xi)=2 \lambda_0$ and $c_2(\xi)$, such that $c^2_1(\xi) -4 \, c_2(\xi)=p_1(E)$.
We denote by $\mathfrak{M}(c_1,c_2)$ the moduli space of rank-two vector bundles $\xi$ over $\mathbb{C}\mathrm{P}^2$ with Chern classes $c_1, \, c_2$.
It is known that $\mathfrak{M}(c_1,c_2)$ only depends on the  discriminant $c_1^2 - 4 c_2$ with the  discriminant being negative for stable bundles.
The bundle $\xi$ can be reduced to an $\mathrm{SU}(2)$-bundle if and only if $c_1(\xi)=0$ and a $\mathrm{SO}(3)$-bundle, which does not
arise as the associated bundle for the adjoint representation of a $\mathrm{SU}(2)$-bundle, satisfies $w_2(E) \not = 0$.

From now on, we will restrict ourselves to $w_2(E) = 0$, i.e.,  the case of $\mathrm{SU}(2)$-bundles and where $c_1(\xi)=0$ and $c_2(\xi)=k \operatorname{H}^2$ with $k \in \mathbb{N}$ .
The moduli space of anti-selfdual irreducible $\mathrm{SU}(2)$-connections with $c_2(\xi)[\mathbb{C}\mathrm{P}^2]=k$ modulo gauge transformations is then
the smooth, projective variety $\mathfrak{M}(0,k)$ of dimension $2d_k= 8 k - 6$ \cite{Stromme}. The generating function (\ref{donwi}) can be described as follows
\begin{equation}
\label{donwi2}
 Z_{\mathrm{DW}}(p, \kappa)= \sum_{m,n \ge 0} \Phi_{m,n} \; \frac{p^m}{m!} \, \frac{\kappa^n}{n!} \;,
\end{equation}
where $S = \kappa \, \operatorname{h}$.  Here $\Phi_{m,n}$ is the intersection number obtained by evaluating the top-dimensional cup product
of the $m$th power of a universal four-form and the $n$th power of a two-form on the fundamental class of the Uhlenbeck compactification of $\mathfrak{M}(0,k)$
such that $4m+2n=8k-6$ with $k\in \mathbb{N}$. Thus, for dimensional reasons we have $\Phi_{m,n}=0$ for $2m + n \not \equiv 1 \pmod{4}$

\subsection{The work of G\"ottsche and his collaborators}

The work of G\"ottsche and his collaborators \cite{Goettsche,GoettscheNakajimaYoshioka}  gives a closed expression
for the $\mathrm{SU}(2)$ Donaldson invariants for the complex projective plane.

Using the blowup formula for the Donaldson invariants, G\"ottsche \cite{Goettsche} derived a closed formula expression for $\Phi_{m, n}$
assuming the truth of the Kotschick-Morgan Conjecture.  Recently, G\"ottsche, Nakajima, Hiraku,  and Yoshioka \cite{GoettscheNakajimaYoshioka} have unconditionally
proved these formulas. His work was based on earlier work with Ellingsrud \cite{EllingsrudGoettsche} and Zagier \cite{GoettscheZagier}, and it extended
the results previously obtained by Kotschick and Lisca \cite{KotschickLisca} up to an overall sign convention. We state \cite[Thm.~3.5, (1)]{Goettsche}  using the original
sign convention of  \cite{EllingsrudGoettsche,KotschickLisca}. We write the result in terms of the Jacobi theta-functions $\vartheta_2, \vartheta_3, \vartheta_4$. In this way, we obtain a closed formula expression for $\Phi_{m, n}$, which we shall later show equals the
the Moore-Witten prediction based on the $u$-plane integral.
\begin{theorem}[G\"ottsche \cite{Goettsche}]
\label{Goettsche_thm}
Assuming the notation and hypotheses above, then
we have that
the only non-vanishing coefficients
in the generating function in (\ref{donwi2})
satisfy
\begin{equation}\label{Goettsche}
\begin{split}
 \Phi_{m, 2n+1} = \sum_{l=0}^n \sum_{j=0}^l & \; \frac{(-1)^{n+j+1} \, 2^{2n-3l+4}}{3^{l}} \; \frac{(2n+1)!}{(2n-2l+1)! \; j! \; (l-j)!} \\
 \times & \operatorname{Coeff}_{q^0} \left( \frac{\vartheta_4^8(\tau) \, \left[ \vartheta_2^4(\tau) + \vartheta_3^4(\tau)\right]^{m+j}}
{\left[ \vartheta_2(\tau) \, \vartheta_3(\tau)\right]^{2m+2n+5}}
 \; E^{l-j}_2(\tau) \; K_{2(n-l)}(\tau) \right) \;,
\end{split}
\end{equation}
where $m,n \in \mathbb{N}_0$, $\operatorname{Coeff}_{q^0}$ is the constant term of a series expansion in $q=\exp{(2\pi i \tau)}$.
The series $K_t(\tau)$ is
\begin{equation}
\label{K_function}
 K_{t} (\tau):= q^{\frac{1}{8}}
 \sum_{\beta=1}^\infty \sum_{\alpha=\beta}^\infty (-1)^{\alpha+\beta} \; (2\alpha+1) \; \beta^{t+1} \; q^{\frac{\alpha(\alpha+1) -\beta^2}{2}} \;.
\end{equation}
\end{theorem}
\begin{proof}
The following table summarizes the quantities used by G\"ottsche \cite[Thm.~3.5, (1)]{Goettsche} and in this article:

\medskip
\begin{center}
{\footnotesize
\begin{tabular}{|l|l|c|l|l|}
\cline{1-2}\cline{4-5}
 G\"ottsche  & Present Paper & \qquad \qquad & G\"ottsche  & Present Paper \\
\cline{1-2}\cline{4-5}
 $z$ & $S$ &&  $\theta(\tau) $                                        & $\vartheta_4(\tau)$ \\ \ & \ && \ & \\
 $x$ & $p$ &&  $f(\tau)      $                                        & $\frac{1}{2\sqrt{i}} \vartheta_2(\tau) \, \vartheta_3(\tau)$ \\ \ & \ &&  \ & \\
 $n$ & $2\beta +1, \; \beta \ge 0$ &&  $ \frac{\Delta^2(2\tau)}{\Delta(\tau)\,\Delta(4\tau)}$ & $ - 16 \frac{\vartheta^8_4(\tau)}{\left[ \vartheta_2(\tau) \, \vartheta_3(\tau) \right]^4}$ \\ & \ &&\  \ & \\
 $a$ & $2\alpha, \; \alpha \ge \beta+1$ && $ G_2(2\tau) $                                         & $-\frac{1}{24} \, E_2(\tau)$\\ \ & \ && \ & \\
 $\tau   $ & $\frac{\tau-1}{2}$ &&  $ e_3(2\tau) $ & $\frac{1}{12} \left[ \vartheta_2^4(\tau) + \vartheta_3^4(\tau) \right] $\\ \ &  \ &&\ &  \\
 $q  $ & $-q^\frac{1}{2}$ &&  $ \frac{-3i \, e_3(2\tau)}{f(\tau)^2}$  & $\frac{\vartheta_2^4(\tau) + \vartheta_3^4(\tau)}{\left[ \vartheta_2(\tau) \, \vartheta_3(\tau) \right]^2} $ \\ \ & \ &&\ &  \\
\cline{1-2}\cline{4-5}
\end{tabular}
}
\end{center}

\medskip
\noindent
We use
\begin{eqnarray*}
 && \left( \frac{n}{2} \, \frac{\sqrt{i}}{f(\tau)} \right)^{2(n-l)} \, \left( - \frac{i}{2 \,f(\tau)^2} (2 \, G_2(2\tau)+ e_3(2\tau))\right)^l \\
& = & \frac{(-1)^{n+l}}{2^l \, 3^l} \, (2\beta+1)^{2(n-l)} \, \frac{\left(-E_2(\tau) + \left[\vartheta_2^4(\tau) + \vartheta_3^4(\tau) \right]\right)^l}
{\left[ \vartheta_2(\tau) \, \vartheta_3(\tau) \right]^{2n}} \;.
\end{eqnarray*}
An expansion of the exponential in
\cite[Thm.~3.5, (1)]{Goettsche} then yields (\ref{Goettsche}).
\end{proof}

\subsection{The $u$-plane integral and the work of Moore and Witten}
Here we recall the theory of the $u$-plane and the work of Moore and Witten.

\label{uplane}
From now on we will assume that $(X,g)$ is a smooth, compact, oriented, simply connected Riemannian four-manifold without boundary
and $b_2^+=1$. The $u$-plane integral $Z$ is a generating function in the variables $p$ and $\kappa$
whose coefficients are the integrals
of certain modular forms over the fundamental domain of the group $\Gamma_{0}(4)$. It  depends on the period point $\omega$, the lattice $H_2(X,\mathbb{Z})$ together with the intersection form $(.\, ,.)$, the second Stiefel-Whitney classes
of the gauge bundle $w_2(E)$, and the tangent bundle $w_2(X)$, whose integral liftings are denoted by $2\lambda_0$ and $w_2$ respectively. The $u$-plane integral is non-vanishing only for manifolds with $b_2^{+}=1$.
The explicit form of $Z$ for simply connected four-manifolds was first introduced in \cite{MooreWitten}.
For the reader's convenience, we quickly review the explicit construction of the $u$-plane in this
chapter. Our approach to the $u$-plane integral, as well as its normalization, closely follows the approach in \cite{LabastidaLozano,LozanoMarino,MarinoMoore}.

We will denote the self-dual and anti-self-dual projections of any two-form $\lambda\in H^2(X,\,\mathbb{Z})+\lambda_0$
by $\lambda_{+}=(\lambda,\omega)\omega$ and $\lambda_{-}=\lambda-\lambda_{+}$ respectively.
We first introduce the integral
\begin{equation}
\label{newcpione}
\mathcal{G}(\rho)= \int_{\Gamma_0(4) \backslash \mathbb{H}}^{\text{reg}}
\dfrac{dx dy}{y^{\frac{3}{2}}} \; \widehat f (p,\kappa) \; \bar\Theta(\xi) \;.
\end{equation}
In this expression $\widehat f(p,\kappa)$ is the almost holomorphic
modular form given by
\begin{equation}
\label{nwcnvpione}
\widehat f(p,\kappa)  =
\dfrac{\sqrt{2}}{64 \pi } \dfrac{\vartheta_4^{\sigma}}{h^3 \cdot f_{2}} \;
 e^{2 \, p \, u + S^2 {\widehat T}},
\end{equation}
where $\sigma$ is the signature of $X$ and $S^2=(S,S)=\sum_{i,j} \kappa^i \kappa^j (\mathrm{s}_i,\mathrm{s}_j)$. Also, $\bar{\Theta}$ is the Siegel-Narain theta function
\begin{equation}
\label{siegnar}
\begin{split}
\bar \Theta(\xi)\,   = \, &\exp\left\lbrack \frac{\pi}{2\,y} \Big(\bar{\xi}_{+}^2-\bar{\xi}_{-}^2\Big)\right\rbrack \\
\times  \sum_{\lambda\in H^2+ \lambda_0 }
& \exp\Big\lbrack - i \pi \bar\tau (\lambda_+)^2 - i \pi   \tau(\lambda_-  )^2
- 2 \pi i\, (\lambda,\bar\xi) + \pi i\, (\lambda ,  w_2)\Big\rbrack,
\end{split}
\end{equation}
where $\bar{\xi}= \bar{\xi}_+ + \bar{\xi}_-$, $\bar{\xi}_+=\rho \, y \, h \, \omega$,
$\bar{\xi}_-= S_-/(2\pi h)$, and $\rho \in \mathbb{R}$.
The Siegel-Narain theta function only depends on the lattice data $(H^2(X), \omega, \lambda_0, w_2)$.
We have denoted the intersection form in two-cohomology by $(.\, , .)$,  and we
used Poincar\'e duality to convert cohomology classes into homology classes.
In the above expressions, $u$, $T$, $h$, and $f_{2}$
are the modular forms defined as follows:
\begin{equation}
\label{qforms}
\begin{array}{rclcrcl}
u&=&  \dfrac{\vartheta_2^4+\vartheta_3^4}{2 \, (\vartheta_2\vartheta_3)^2} \;, &\qquad&
h&=& \frac{1}{2} \, \vartheta_2 \, \vartheta_3 \;, \\ [2ex]
T&=&-\dfrac{1}{24} \left(  \dfrac{E_2}{h^2} - 8 \, u\right) \;, &&
f_{2}&=&\dfrac{\vartheta_2 \, \vartheta_3}{2 \, \vartheta_4^8} \;.
\end{array}
\end{equation}
Note that $T$ does not transform well under modular transformations, due to the presence of the second normalized Eisenstein series $E_2=E_2(\tau)$.
Therefore, in  (\ref{nwcnvpione})
we have used the related form $\widehat{T} = T + 1/(8\pi y h^{2})$
which is not holomorphic but transforms well under modular transformations.
We also define the related holomorphic function $f(p, \kappa)$
as in  (\ref{nwcnvpione}), but with $T$ instead of $\widehat{
T}$. The $u$-plane integral is defined to be
\begin{equation}
\label{u-plane_integral}
 Z\Big(X,\omega,\lambda_0,w_2\Big) =\left. \left\lbrack (S,\omega) + 2  \dfrac{d}{d\rho} \right\rbrack\right|_{\rho=0} \; \mathcal{G}(\rho)\;.
\end{equation}
If there is no danger of confusion we suppress the arguments $(X,\omega,\lambda_0,w_2)$ of $Z$.

\medskip
\noindent
{\it Two remarks.}

\noindent 1) This regularized $u$-plane integral can be thought of as the regularized Peterson inner product
of two half-integral weight modular forms on $\Gamma_0(4)$, where the regularization is obtained by integrating
over the truncated fundamental domain for $\Gamma_0(4)$ where neighborhoods of the cusps are removed.

\noindent 2) Definition (\ref{u-plane_integral}) agrees with the definition given
in \cite{LozanoMarino}. However, compared to the original definition in \cite{MooreWitten}, a factor of
$\exp{[2 \pi i (\lambda_0,\lambda_0) + \pi i (\lambda_0, w_2)]}$ is missing. For the case considered in this
article, this factor is equal to one.

\medskip
The regularization procedure applied in the definition of the integral
(\ref{newcpione}) was described in detail in \cite{MooreWitten}. It defines
a way of extracting certain contributions for each boundary component near the cusps of $\Gamma_0(4)\backslash\mathbb{H}$.
Since the cusps are located at $\tau=\infty$, $\tau=0$, and $\tau=2$, we obtain
$Z_{\mathrm{u}}$ as the sum of these contributions from the cusps:
\begin{equation}
 Z_{\mathrm{u}} = Z_{\tau=0} + Z_{\tau=2} + Z_{\tau=\infty} \;.
\end{equation}

We now apply the construction of the $u$-plane integral to $X=\mathbb{C}\mathrm{P}^2$.
We denote the integral lifting of $w_2(E)$ by $2\lambda_0=a \, \mathrm{H} \in H^2(\mathbb{C}\mathrm{P}^2,\mathbb{Z})$,
and the integral lifting of $w_2(\mathbb{C}\mathrm{P}^2)$ by $w_2=- b \, \mathrm{H}\in H^2(\mathbb{C}\mathrm{P}^2,\mathbb{Z})$.
The following lemma then follows immediately from the definition:
\begin{lemma}
On $X=\mathbb{C}\mathrm{P}^2$ let $\omega=\mathrm{H}$ be the period point
of the metric. Let $2\,\lambda_0 = a \, \mathrm{H}$ with $a \in \{0,1\}$ be an integral lifting
of $w_2(E)$. For $(X,\omega,\lambda_0,w_2=- \, \mathrm{H})$,
the Siegel-Narain theta function is
\begin{equation}
\label{siegnar_CP2}
\begin{split}
\bar \Theta  =  \exp{\left(\frac{\pi}{2\,y} \, \bar{\xi}_+^2\right)} \; \overline{\vartheta_{a1}\Big((\xi_+,\mathrm{H}) \Big| \tau\Big)},\;
\end{split}
\end{equation}
where $\bar\xi=\bar\xi_+=\rho \, y\, h \, \omega$.
\end{lemma}
\noindent
It was shown in \cite{MooreWitten} that for $\sigma=b_2^+-b_2^-=1$ and any value of $a$, we have
\begin{equation}
Z_{\tau=0} = Z_{\tau=2} =0 \;,
\end{equation}
and so $Z_{\mathrm{u}} = Z_{\tau=\infty}$.
We restrict ourselves to the case $a=0$, i.e., the case of $\mathrm{SU}(2)$-bundles on $\mathbb{C}\mathrm{P}^2$.
The $u$-plane integral in (\ref{u-plane_integral}) can be expanded as follows
\begin{equation}
 Z_{\tau=\infty}  = \sum_{m,n \in \mathbb{N}_0} \dfrac{p^m}{m!} \; \dfrac{\kappa^{2n+1}}{(2n+1)!} \; \mathrm{D}_{m,n},
\end{equation}
where
\begin{equation}
\label{coefficients_preint}
\mathrm{D}_{m,n} := -\frac{\sqrt{2}}{32\pi} \sum_{l=0}^n  \int^{\text{reg}}_{\Gamma_0(4)\backslash \mathbb{H}}  \frac{dx \, dy}{y^{\frac{3}{2}}} \; R_{mnl} \; \widehat{E}_2^l\;  \, \overline{\vartheta_{01}(0| \tau)}  \;.
\end{equation}
For $m,n \in \mathbb{N}_0$ and $0 \le l \le n$ we have set
\begin{equation}
\label{R_mnk}
 R_{mnl} := (-1)^{l+1} \; \dfrac{(2n+1)!}{l! \, (n-l)!} \; \dfrac{2^{m-3l-1}}{3^n} \;
 \dfrac{\vartheta_4 \cdot u^{m+n-l}}{h^{3+2l} \cdot f_{2}},
\end{equation}
where $u$, $h$, and $f_2$ were defined in (\ref{qforms}).
To evaluate the regularized $u$-plane integral we introduce the non-holomorphic modular form
$Q_{ab}(\tau)=Q^+_{ab}(\tau)+Q^-_{ab}(\tau)$  of weight $3/2$
such that
\begin{equation}\label{total_derivative}
 8 \, \sqrt{2} \pi \, i \; \frac{d}{d\bar{\tau}} \; Q_{ab}\left(\tau\right) = y^{-\frac{3}{2}} \, \overline{\vartheta_{ab}(0 | \tau)},
 \end{equation}
 where $a$ or $b$ must be zero.
 These non-holomorphic modular forms
were constructed by Zagier \cite{Zagier} and reviewed in \cite{MooreWitten}.
The holomorphic parts of Zagier's weight 3/2 Maass-Eisenstein series,
which first arose \cite{HirzebruchZagier} in connection with
intersection theory for certain Hilbert modular surfaces, are generating functions for Hurwitz class numbers.
The holomorphic part of Zagier's weight 3/2 Maass-Eisenstein series
is the generating function for Hurwitz class numbers.
They have series expansions of the form
\begin{equation}\label{Qplus_0}
 \begin{split}
  Q_{10}^+\left(\tau\right) &= \frac{1}{q^{\frac{1}{8}}} \; \sum_{l > 0} \mathcal{H}_{4l-1} \; q^{\frac{l}{2}} \;,\\
  Q_{00}^+\left(\tau\right) &= \sum_{l \ge 0} \mathcal{H}_{4l} \; q^{\frac{l}{2}},
 \end{split}
\end{equation}
where $\mathcal{H}_{\alpha}$ are the Hurwitz class numbers. The first nonvanishing Hurwitz class numbers are as follows:
\begin{equation*}
\begin{array}{|c|c|c|c|c|c|c|c|}\hline
 \mathcal{H}_0 & \mathcal{H}_3 & \mathcal{H}_4 & \mathcal{H}_7 & \mathcal{H}_8 & \mathcal{H}_{11} & \mathcal{H}_{12} & \dots \\
 \hline &&&&&&& \\[-2ex]
 - 1/12 & 1/3 & 1/2 & 1 & 1 & 1 & 4/3 & \dots\\ \hline
\end{array}
\end{equation*}
The non-holomorphic parts have series expansions of the form
\begin{equation}\label{Qminus_0}
\begin{split}
 Q^{-}_{10}\left(\tau\right) & =  \frac{1}{8\sqrt{2\pi}} \; \sum_{l = - \infty}^\infty \; (l+\frac{1}{2}) \cdot \Gamma\left( -\frac{1}{2} , 2 \, \pi \, \left(l+\frac{1}{2}\right)^2 \, y \right) \; q^{-\frac{(l+1/2)^2}{2}} \;,\\
 Q^{-}_{00}\left(\tau\right) & =  \frac{1}{8\sqrt{2\pi}} \; \sum_{l = - \infty}^\infty \; l \cdot \Gamma\left( -\frac{1}{2} , 2 \, \pi \, l^2 \, y \right) \; q^{-\frac{l^2}{2}} \;,
\end{split}
\end{equation}
where $\Gamma(3/2,x)$ is the incomplete gamma function
\begin{equation}
 \Gamma(\alpha,x) = \int_x^\infty e^{-t} \; t^{\alpha-1} \; dt \;.
\end{equation}
The forms $Q_{10}$ and $Q_{00}$ combine to form a weight $3/2$ form for the modular group.
As explained in \cite{MooreWitten} the form $Q_{01}(\tau)=Q_{00}(4\tau)-Q_{10}(4\tau)+\frac{1}{2}Q_{00}(\tau+1)$
is modular for $\Gamma_0(4)$ of weight $3/2$. We write the holomorphic part as
\begin{equation}\label{Q_01}
  Q_{01}^+\left(\tau\right) = \sum_{n \ge 0} \mathcal{R}_{n} \; q^{\frac{n}{2}}  \;.
\end{equation}
The first nonvanishing coefficients in the series expansion are as follows:
\begin{equation*}
\begin{array}{|c|c|c|c|c|c|}\hline
 \mathcal{R}_0 & \mathcal{R}_1 & \mathcal{R}_2 & \mathcal{R}_3 & \mathcal{R}_4 & \dots \\
 \hline &&&&& \\[-2ex]
 - 1/8 & -1/4 & 1/2 & -1 & 5/4 & \dots\\ \hline
\end{array}
\end{equation*}
All non-holomorphic parts have an exponential decay since
\begin{equation}
\label{decay}
  \Gamma\left( \alpha , t\right) = t^{\alpha-1} \; e^{-t} \; \left( 1 + O(t^{-1}) \right)
\qquad (t \to \infty) \;.
\end{equation}
The following lemma was proved in \cite[(9.18)]{MooreWitten}:
\begin{lemma}\label{EkQ}
The weakly holomorphic function
\begin{equation}
\label{EkQab}
 \mathcal{E}^l \left[ Q_{01} \right] = \sum_{j=0}^l (-1)^j \; \binom{l}{j} \; \frac{\Gamma\left(\frac{3}{2}\right)}{\Gamma\left(\frac{3}{2}+j\right)} \; 2^{2j} \; 3^j
 \; E_2^{l-j}(\tau) \; \left(q \, \frac{d}{dq} \right)^j Q_{01}\left(\tau\right)
\end{equation}
is modular for $\Gamma_0(4)$ of weight $2l+3/2$ and satisfies
\begin{equation}
 8\,\sqrt{2} \pi \, i \; \frac{d}{d\bar{\tau}}\; \mathcal{E}^l \left[ Q_{01} \right]  =
y^{-\frac{3}{2}} \, \widehat{E}_2^l(\tau) \; \overline{\vartheta_{01}(0 | \tau)}  \;.
\end{equation}
\end{lemma}

\subsubsection{The evaluation of the $u$-plane integral.}
It was shown in \cite{MooreWitten} that the cusp contribution at $\tau=\infty$ to the
regularized $u$-plane integral can be evaluated as follows: in (\ref{coefficients}) we integrate by parts using the modular forms constructed in Lemma \ref{EkQ}, i.e., we rewrite an integrand $f$ as a total derivative using
\begin{equation*}
 dx\wedge dy \; \partial_{\bar{\tau}} f  = \frac{1}{2} \, dx \wedge dy \; \left( \partial_{x} + i \, \partial_{y} \right) \, f
= -\frac{i}{2} \, d\Big(  f \, dx + i \, f \, dy \Big) \;.
\end{equation*}
We carry out the integral along the boundary $x=\re{(\tau)} \in [0,4]$ and $y \gg 1$ fixed. This extracts the constant term
coefficient. We then take the limit $y \to \infty$. Since all non-holomorphic parts
have an exponential decay, the non-holomorphic dependence drops out. The following expression for
the $u$-plane integral was obtained for the gauge group $\mathrm{SU}(2)$ in \cite{MooreWitten}. Additional information about the evaluation of the $u$-plane integral as well as the geometry of the
Seiberg-Witten curve can be found in \cite{M2,M}.
\begin{theorem}
\label{evaluation1}
On $X=\mathbb{C}\mathrm{P}^2$, let $\omega=\mathrm{H}$ be the period point of the metric.
For $(X,\omega,\lambda_0=0,w_2=- \mathrm{H})$,
the $u$-plane integral in the variables $p \, \mathrm{x} \in H_0(X, \mathbb{Z})$,
$S=\kappa \,\mathrm{h} \in H_2(X, \mathbb{Z})$ is
\begin{equation}
\label{generating_function}
 Z_{\mathrm{u}}=Z_{\tau=\infty} = \sum_{m,n \in \mathbb{N}_0} \dfrac{p^m}{m!} \dfrac{\kappa^{2n+1}}{(2n+1)!} \; \mathrm{D}_{m,n},
\end{equation}
where
\begin{equation}
\label{coefficients}
\mathrm{D}_{m,n} =  \sum_{l=0}^n  \; \operatorname{Coeff}_{q^0}  \Big(  R_{mnl} \; \; \mathcal{E}^l[Q_{01}^+(\tau)] \Big)
\end{equation}
and where $R_{mnl}$ and $\mathcal{E}^l[Q_{01}(\tau)]$ are defined in
(\ref{R_mnk}) and (\ref{EkQab}) respectively.
\end{theorem}

\noindent
For concreteness, we list the first nonvanishing coefficients of the generating function
in Theorem \ref{evaluation1}, i.e., if $m+n = 2(k-1)$ for some $k \in \mathbb{N}$.
\medskip
\begin{center}
\begin{tabular}{|l|ll||l|l|}
\hline
$k$ & $m$ &  $n$ & $\mathrm{D}_{m,n}$ & $\mathrm{D}_{m,n}$ \\
\hline
 \raisebox{-0.5ex}{$1$}  &  \raisebox{-0.5ex}{$0$}  &  \raisebox{-0.5ex}{$0$}  & \raisebox{-0.5ex}{$-\frac{3}{2}$}              & \raisebox{-0.5ex}{$-\frac{1}{2} \,  \mathcal{R}_1 + 13 \, \mathcal{R}_0$} \\ [1.5ex]
 $2$ & $0$ & $2$ &$\phantom{-}1$& $-2 \, \mathcal{R}_2 + 7 \, \mathcal{R}_1 -30 \, \mathcal{R}_0$\\ [1ex]
 $2$ & $1$ & $1$ &$-1$&$-\frac{1}{4} \, \mathcal{R}_2 + \frac{1}{2} \, \mathcal{R}_1 + 6 \, \mathcal{R}_0$\\[1ex]
 $2$ & $2$ & $0$ &$-\frac{13}{8}$&$-\frac{1}{32} \, \mathcal{R}_2 - \frac{7}{16} \, \mathcal{R}_1 + \frac{55}{4} \, \mathcal{R}_0$\\ [1ex]
\hline
 \end{tabular}
\end{center}

\subsection{Criterion for proving Theorem~\ref{MainTheorem}}
Here we combine the results of the previous two subsections to obtain
a criterion for proving Theorem~\ref{MainTheorem}.

From a physics point of view, at a high energy scale, the $\mathrm{SU}(2)$-Donaldson theory
is described by the low energy effective field theory.  Thus, the cuspidal contributions to the
generating function of the low energy effective field theory should be equal to the generating
function of the $\mathrm{SU}(2)$-Donaldson theories
The conjecture is equivalent to the assertion that the generating functions $Z_{\mathrm{DW}}$ in (\ref{donwi2}) and
$Z_{\mathrm{u}}$ in (\ref{generating_function}) are equal.  This amounts to proving that for all $m,n \in \mathbb{N}_0$ we
have
\begin{equation}
\label{first}
\Phi_{m,2n+1} = \mathrm{D}_{m,n} \;.
\end{equation}
\noindent
In particular, the coefficients in (\ref{first}) vanish for $m+n \equiv 1 \pmod{2}$.
We will prove (\ref{first}) by proving:
\begin{theorem}\label{Criterion}
Theorem~\ref{MainTheorem} is equivalent to the vanishing of constant
terms, for every pair of non-negative integers $m$ and $n$, of
the series
\begin{equation}
\begin{split}
 \sum_{l=0}^n \sum_{j=0}^l & \; (-1)^{j+1}  \; \frac{(2n+1)!}{(n-l)! \; j! \; (l-j)!}
 \frac{\vartheta_4^8(\tau) \, \left[ \vartheta_2^4(\tau) + \vartheta_3^4(\tau)\right]^{m}}
{\left[ \vartheta_2(\tau) \, \vartheta_3(\tau)\right]^{2m+2n+4}}
 \; E^{l-j}_2(\tau)  \\
\times & \left[ \frac{(-1)^{n} \, 2^{2n-3l+4}}{3^{l}} \; \frac{(n-l)!}{(2n-2l+1)!}
 \; \dfrac{\left[ \vartheta_2^4(\tau) + \vartheta_3^4(\tau)\right]^{j}}{\vartheta_2(\tau) \, \vartheta_3(\tau)} \; K_{2(n-l)}(\tau) \right. \\
- & \left. \frac{(-1)^{l}\, 2^{2j-n+3}}{3^{n-j}} \;
 \frac{\Gamma\left(\frac{3}{2}\right)}{\Gamma\left(j+\frac{3}{2}\right)} \;
 \vartheta_4(\tau) \, \left[\vartheta^4_2(\tau) + \vartheta^4_3(\tau)\right]^{n-l} \; \left( q \frac{d}{dq}\right)^j  Q_{01}^+\left(\tau\right) \right] \;,
\end{split}
\end{equation}
where the series $K_t(\tau)$ are defined in (\ref{K_function}).
\end{theorem}

\section{The proof of Theorem~\ref{MainTheorem}}\label{Section3}
 Here we prove Theorem~\ref{MainTheorem} by using the theory of non-holomorphic modular forms and meromorphic Jacobi forms
to check the condition in Theorem~\ref{Criterion}. To this end, we recall the important $q$-series
\begin{equation}
K_{t} (\tau):= q^{\frac{1}{8}}
 \sum_{\beta=1}^\infty \sum_{\alpha=\beta}^\infty (-1)^{\alpha+\beta} \; (2\alpha+1) \; \beta^{t+1} \; q^{\frac{\alpha(\alpha+1) -\beta^2}{2}}
\end{equation}
from (\ref{K_function}). In the following section we relate $K_{2t}(\tau)$ to derivatives of important power series.

\subsection{$q$-series identities}
Here we begin with the following elementary identity.

\begin{proposition}
\label{K_as_AL_Sum}
Let $\rho=e^{2\pi i u}$ and $\omega=e^{2\pi i v}$, and let $D_z=\frac{1}{2 \pi i}\frac{d}{dz},$ where $z$ is one of $u$, $v$, or $\tau$. Then we have that
\begin{displaymath}
K_{2t}(8\tau)= 2^{-2t-1}\left.D_u^{2t+1}D_v \sum_{n\in \Z}\frac{(-1)^n\omega^{2n+1}q^{(2n+1)^2}}{1-\rho^2\omega^2q^{8n+4}}\right|_{u=v=0}.
\end{displaymath}
\end{proposition}

\begin{proof}
By rearranging terms, it is not difficult to see that the summation on the right (prior to taking derivatives) is equal to
\begin{equation}
\begin{split}
\sum_{n\geq 0}&\frac{(-1)^n\omega^{2n+1}q^{(2n+1)^2}}{1-\rho^2\omega^2q^{4(2n+1)}}  -\rho^{-2}\omega^{-2}q^{4(2n+1)}\frac{(-1)^n\omega^{-(2n+1)}q^{(2n+1)^2}}{1-\rho^{-2}\omega^{-2}q^{4(2n+1)}}\\
&= \sum_{n\geq 0}(-1)^n\omega^{2n+1}q^{(2n+1)^2} \\
&\quad+\sum_{n\geq 0}\sum_{m\geq1}(-1)^n\left(\omega^{2n+1+2m}\rho^{2m}+\omega^{-(2n+1+2m)}\rho^{-2m}\right)q^{(2n+1)^2+4m(2n+1)}.
\end{split}
\end{equation}

We then set $\alpha=n+m$, and $\beta=m$. After applying the derivatives, evaluating at $u=v=0$, and factoring out the powers of $2$, this becomes
\begin{equation}
 \sum_{\beta=1}^\infty \sum_{\alpha=\beta}^\infty (-1)^{\alpha+\beta} \; (2\alpha+1) \; \beta^{2t+1} \; q^{4\alpha^2+4\alpha -4\beta^2+1} =K_{2t} (8\tau).
\end{equation}
\end{proof}

The summation in the right hand side of the equation in Proposition~\ref{K_as_AL_Sum} is in the form of an Appell-Lerch function. In the next section, we show how to write this in terms of Zwegers's $\mu$-function, from which we can infer its modularity properties.

\subsection{Work of Zwegers}

In his Ph.D. thesis on mock theta functions \cite{Zwegers}, Zwegers constructs weight 1/2 harmonic weak Maass forms by making use of the transformation properties of functions which were investigated earlier by Appell and Lerch. Here we briefly recall some of his results.

For $\tau$ in $\H$ and $u,v \in \C \setminus (\Z\tau+\Z)$, Zwegers defines the function
\begin{equation}\label{Def_of_Mu}
\mu(u,v;\tau):=\frac{\rho^{1/2}}{\theta(v;\tau)}\cdot \sum_{n\in \Z}\frac{(-\omega)^nq^{n(n+1)/2}}{1-\rho q^n},
\end{equation}
where $\rho=e^{2\pi i u}$ and $\omega=e^{2\pi i v}$ as above, and
\begin{equation}
\theta(v;\tau):=\sum_{\nu\in \Z+\frac12}(-1)^{\nu-\frac12}\omega^{\nu}q^{\nu^2/2}.
\end{equation}
Zwegers's (see Section 1.3 of \cite{Zwegers}) proves that $\mu(u,v,\tau)$ satisfies the following important properties.
\begin{lemma}\label{Mu_Transformations}
Assuming the notation above, we have that\\\\
\begin{tabular}{lccc}
(1)&$\mu(u,v;\tau) $&$=$&$ \mu(v,u,\tau),$\\\\
(2)&$\mu(u+1,v,\tau)$&$=$&$-\mu(u,v;\tau),$\\\\
(3)&$\rho^{-1}\omega q^{-\frac12}\mu(u+\tau,v;\tau)$&$=$&$-\mu(u,v;\tau)+\rho^{-\frac12}\omega^{\frac12}q^{-\frac18},$\\\\
(4)&$\mu(u,v;\tau+1)$&$=$&$\zeta_8^{-1}\mu(u,v;\tau)\quad(\zeta_N:=e^{2\pi i/N})$\\\\
(5)&$(\tau/i)^{-\frac12}e^{\pi i (u-v)^2/\tau}\mu(\frac u\tau,\frac v\tau;-\frac1\tau)$&$=$&$-\mu(u,v;\tau)+\frac12h(u-v;\tau),$\\\\
\end{tabular}\\
where
\[h(z;\tau):=\int_{-\infty}^{\infty}\frac{e^{\pi i x^2\tau-2\pi xz}dx}{\cosh \pi x}\]
\end{lemma}
\remark{ The integral $h(z,\tau)$ is known as a {\it Mordell integral}.}

Lemma \ref{Mu_Transformations} shows that $\mu(u,v;\tau)$ is nearly a weight 1/2 Jacobi form, where $\tau$ is the modular variable. Zwegers then uses $\mu$ to construct weight 1/2 harmonic weak Maass forms. He achieves this by modifying $\mu$ to obtain a function $\widehat{\mu}$ which he then uses as a building block for such Maass forms. To make this precise, for $\tau \in \H$ and $u\in \C$, let $c:= \im(u)/\im(\tau)$, and let
\begin{equation}
R(u;\tau) := \sum_{\nu\in \Z+\frac12} (-1)^{\nu-\frac12}\left\{ \sgn(\nu) -E\left( (v+c)\sqrt{2\im(\tau)}\right)\right\}e^{-2\pi i\nu u} q^{-\nu^2/2},
\end{equation}
where $E(z)$ is the odd function
\[E(z):= 2\int_0^ze^{-\pi u^2}du.\]
Using $\mu$ and $R$, we let
\begin{equation}
\widehat{\mu}(u,v;\tau) := \mu(u,v;\tau)-\frac12 R(u-v;\tau).
\end{equation}

Zwegers's construction of weight 1/2 harmonic weak Maass forms depends on the following theorem (see Section 1.4 of \cite{Zwegers}).

\begin{theorem}\label{Mu_hat_Transformations} Assuming the notation above, we have that\\

\begin{tabular}{llcll}
(1)&$\widehat \mu(u,v;\tau)$ &$=$&$ \widehat\mu(v,u,\tau),$&\\\\
(2)& $\widehat \mu(u+1,v,\tau)$&$=$&$\rho^{-1}\omega q^{-\frac12}\mu(u+\tau,v;\tau)$&$ \hskip-.8in= -\widehat\mu(u,v;\tau),$\\\\
(3)& $\zeta_8^{-1} \widehat\mu(u,v;\tau+1)$ &$=$&$ -(\tau/i)^{-\frac12}e^{\pi i (u-v)^2/\tau}\widehat\mu(\frac u\tau,\frac v\tau;-\frac1\tau)$& $\hskip-.8in= \widehat\mu(u,v;\tau),$\\\\
(4)&$\widehat \mu\left(\frac{u}{c\tau+d},\frac{v}{c\tau+d};\frac{a\tau+b}{c\tau+d}\right)$&=&$\chi(A)^{-3}(c\tau+d)^{\frac12}e^{-\pi i c(u-v)^2/(c\tau+d)}\cdot \widehat\mu(u,v;\tau),$&\\\\
\end{tabular}\\
where $A=\begin{pmatrix}a&b\\c&d\end{pmatrix}$, and $\chi(A):= \eta(A\tau)/\left((c\tau+d)^{\frac 12}\eta(\tau)\right).$
\end{theorem}

Theorem \ref{Mu_hat_Transformations} gives the modular transformation properties for $\widehat \mu$. In the following section we will write $K_{2t}$ in terms of $\mu$, and we then use its properties to complete $K_0(\tau)$ as a nonholomorphic modular form on $\Gamma_0(8).$
\subsection{Modularity Properties of $K_0(\tau)$}
We begin with the following proposition.
\begin{proposition}\label{K_in_Mu} We have that
\[K_{2t}(8\tau)=2^{-2t-1}\left.D_u^{2t+1}D_v\left( \rho^{-1}q^{-1}\mu(2u+2v+4\tau,2v;8\tau)\theta(2v;8\tau)\right)\right|_{u=v=0}.\]
Moreover, $\frac{K_0(8\tau)}{\eta^3(8\tau)}$ is the holomorphic part of a weight $3/2$ weak Maass which is modular on $\Gamma_0(8)$, and whose  non-holomorphic part is the period integral of $\Theta_4(\tau)$.
\end{proposition}
\begin{proof}
The first statement follows directly from Proposition~\ref{K_as_AL_Sum} and the definition of the $\mu$ function defined in (\ref{Def_of_Mu}).

To prove the remainder of the proposition, let
\begin{equation}\label{K_hat}
\widehat {K_0}(\tau)=2^{-1}\left.D_uD_v\left( \rho^{-1}q^{-1}\widehat{\mu}(2u+2v+4\tau,2v;8\tau)\theta(2v;8\tau)\right)\right|_{u=v=0},
\end{equation}
so that the holomorphic part of $\widehat {K_0}(\tau)$ is $K_0(8\tau).$ Suppose $A=\begin{pmatrix}a&b\\8c&d\end{pmatrix}\in \Gamma_0(8)$. Then we note $\overline{A}=\begin{pmatrix}a&8b\\c&d\end{pmatrix}\in SL_2(\Z).$ Using the transformation laws for $\widehat\mu$ found in Lemma~\ref{Mu_hat_Transformations}, we have
\begin{equation}
\begin{split}
e^{2\pi i \frac{-u-a\tau-b}{8c\tau+d}} \widehat{\mu}&\left(\frac{2u+2v+4(a\tau+b)}{8c\tau+d},\frac{2v}{8c\tau+b};\frac{a8\tau+8b}{c8\tau+d}\right)\\
= &\chi(\overline{A})^{-3}(-1)^{\frac{a-1}{2}}(8c\tau+d)^{1/2}e^{(2\pi i )\frac{-1}{4}\frac{8cu^2}{8c\tau+d}}\cdot \widehat{\mu}(2u+2v+4\tau,2v,8\tau),
\end{split}
\end{equation}
which is obtained by substituting $u\to \frac{u}{8c\tau+d}$, $v\to \frac{v}{8c\tau+d}$, and $\tau \to \frac{a\tau+b}{8c\tau+d}$ into the expression on the right hand side of (\ref{K_hat}), before taking derivatives. With a little more algebra, we  find that
\begin{equation}\widehat {K_0}\left(\frac{a\tau+b}{8c\tau+d}\right) =\chi(\overline{A})^{-3}(-1)^{\frac{a-1}{2}}(8c\tau+d)^{3}\widehat {K_0}(\tau).
\end{equation}
 Therefore $\widehat {K_0}(\tau)$ is modular on $\Gamma_0(8)$ with weight $3/2$. The non-holomorphic part of $\widehat {K_0}(\tau)$ is
\begin{equation}
\frac{-1}{4}\left.D_uD_vR(2u+4\tau;8\tau)\theta(2v;8\tau) \right|_{u=v=0} = \frac{-}{4}\eta^3(8\tau)\left.D_uR(2u+4\tau;8\tau)\right|_{u=0}.
\end{equation}
After factoring out $\eta^3(8\tau)$, a straightforward calculation gives us that
\begin{equation}\label{K_non-h_part}
\frac{\partial}{\partial \bar\tau}\frac{-1}{4}\left.D_uR(2u+4\tau;8\tau)\right|_{u=0}=\frac{-1}{32\pi y^{\frac{3}{2}}}\overline{\Theta_4(\tau)}.
\end{equation}
\end{proof}


\subsection{The proof of Theorem~\ref{MainTheorem}}

Thanks to Theorem \ref{Criterion}, it suffices to prove that the differences between certain $q$-series have vanishing constant term. We shall derive these conclusions by using differential operators, using methods very similar to those found in Section 8.1 of \cite{M_Ono}. For brevity, we describe the $n = 0$ cases in detail, and then provide general remarks which are required to justify the remaining cases.

 By (\ref{Q_01}) and Proposition~\ref{K_in_Mu}, we have
\begin{equation}
8Q_{01}^+(8q) =  -1 -2q^4+4q^8 -8q^{12}+10q^{16}+ \dots,
\end{equation}
and
\begin{equation}
8\frac{\widehat{K_0}(\tau)}{\eta^3(8\tau)} = 24q^4 + 80q^8+240 q^{12}+528 q^{16}\dots .
\end{equation}
Comparing (\ref{total_derivative}) (with $8\tau$ substituted for $\tau$) and (\ref{K_non-h_part}), we see that both of these are the holomorphic parts of weight 3/2 harmonic weak Maass forms with equal non-holomorphic parts. Therefore, it follows that
\begin{equation}
8\frac{\widehat {K_0}(\tau)}{\eta^3(8\tau)}-8Q_{01}(8\tau) = 1+26q^4+76q^8+248q^{12}+518q^{16}+\dots
\end{equation}
is a modular form. A short calculation shows that
\begin{equation}\label{KQ_as_E2*}
\frac{\widehat {K_0}(\tau)}{\eta^3(8\tau)}-Q_{01}(8\tau)=\frac{E^*(4\tau)}{8\Theta_4(\tau)}
\end{equation}
where $E^*(\tau)$ is the weight 2 Eisenstein series
\begin{eqnarray}
E^*(\tau):=-E_2(\tau)+2E_2(2\tau)=1+24\sum_{n=1}^\infty\sigma_{\mathrm{odd}}(n)q^n,
\end{eqnarray}
and $\sigma_{\mathrm{odd}}(n)$ denotes the sum of the positive odd divisors of $n$. Noting that
 \begin{equation}
\eta^3(8\tau)=\Theta_2(\tau)\Theta_3(\tau)\Theta_4(\tau),
\end{equation}
 where $\Theta_1$, $\Theta_2$, and $\Theta_3$ are defined in (\ref{Thetas}), we can rewrite this as
\begin{eqnarray}
\frac{8\widehat {K_0}(\tau)}{\Theta_2(\tau)\Theta_3(\tau)}-8\Theta_4(\tau)Q_{01}(8\tau)=E^*(4\tau).
\end{eqnarray}

For $n=0$, Theorem \ref{MainTheorem} is equivalent to the claim, for every $m\geq 0$, that the constant term vanishes in the expression

\begin{eqnarray}\label{N_eq_0}
\frac{\Theta_4(\tau)^8(16\Theta_2(\tau)^4+\Theta_3(\tau)^4)^mE^*(4\tau)}{\Theta_2(\tau)^{2m+4}\Theta_3(\tau)^{2m+4}}.
\end{eqnarray}
 In order to verify this claim, we will find if helpful to define
\begin{eqnarray}
Z(q) := \frac{E^*(4\tau)}{\Theta_2(\tau)^2\Theta_3(\tau)^2}.
\end{eqnarray}
which has the derivative \begin{eqnarray}
q\frac{d}{dq}Z(q) = \frac{-2~\Theta_4(\tau)^8}{\Theta_2(\tau)^2\Theta_3(\tau)^2}.
\end{eqnarray}
Here $Z(q)$ is the same as $\widehat{Z_0}(q)$ defined in Section 8.1 of \cite{M_Ono}.
We also note that
\[16\Theta_2(\tau)^4+\Theta_3(\tau)^4=1+24q^4+24q^2+\cdots=E^*(4\tau).\]
Using this notation, (\ref{N_eq_0}) becomes
\[\frac{-1}{2(m+2)}~q\frac{d}{dq} Z(q)^{m+2},\]
which  has a vanishing constant term.

In fact, for each $m,n\geq 0$, we find a similar phenomenon. For every non-negative $k$, define
\begin{equation}
\begin{split}
\mathcal G _{\ell}(q) := \sum_{j=0}^\ell  &\begin{pmatrix}\ell\\j \end{pmatrix}\frac{ (-12)^j E_2(8\tau)^{\ell-j}\Gamma(\frac32)}{\left(\Theta_2(\tau)\Theta_3(\tau)\right)^{2\ell+2}8^j \Gamma\left(\frac 32+j\right)} \\
&\times\left[\frac{(-4)^j 8 K_{2j}(8\tau)}{\Theta_2(\tau)\Theta_3(\tau)} -8\Theta_4(\tau)\left(q\frac{d}{dq}\right)^j Q_{01}^+(8\tau)  \right].
\end{split}
\end{equation}
Using this notation, the criterion given in Theorem \ref{Criterion} is equivalent to the claim that the constant coefficient of
\begin{eqnarray}
\left(q\frac{d}{dq}Z(q)\right) Z(q)^m\sum_{\ell=0}^n  \begin{pmatrix}n\\\ell \end{pmatrix}\left(-Z(q)\right)^{n-\ell}\mathcal G _{\ell}(\tau)
\end{eqnarray}
 is zero for each non-negative $m$ and $n$. It suffices to show that $\mathcal G _{\ell}(q)$ is a polynomial in $Z(q)$. We define $M_0^*(\Gamma_0(8))$ to be the space of modular functions on $\Gamma_0(8)$ which are holomorphic away from infinity, and is a subspace of $\C((q^2))$. One can easily verify that $M_0^*(\Gamma_0(8))$ is precisely the set of polynomials in $Z(q)$. In order to show that $\mathcal G _{\ell}(\tau)$ is in $M_0^*(\Gamma_0(8))$, we first show that a similar function, $\mathcal H_\ell(q)$ is in $M_0^*(\Gamma_0(8))$. We define the function

\begin{equation}
\mathcal H _{\ell}(q):= \frac{\Theta_4(\tau)}{\left(\Theta_2(\tau)\Theta_3(\tau)\right)^{2\ell+2}} \sum_{j=0}^\ell \begin{pmatrix}\ell\\j \end{pmatrix}\frac{\Gamma(\frac32) (-12)^j E_2(8\tau)^{\ell-j}}{\Gamma\left(\frac 32+j\right)8^j} \left(q\frac{d}{dq}\right)^j \frac{E^*(8\tau)}{\Theta_4(\tau)}.
\end{equation}

We can observe that $\mathcal H _{\ell}(q)$ is modular on $\Gamma_0(8)$ with weight $0$ by comparing the summation to the expression $\mathcal E^\ell\left[ \frac{E^*(8\tau)}{\Theta_4(\tau)} \right],$ where the bracket operator
\[
\mathcal E^\ell[f]:=\sum_{j=0}^\ell \begin{pmatrix}\ell\\j\end{pmatrix}\frac{\Gamma\left(\frac32\right)}
{\Gamma\left(\frac32+j\right)}(-12)^j E_2^{\ell-j}(\tau)\left(q\frac{d}{dq}\right)^jf(\tau)
\]
 is defined as in  equation (9.18) of \cite{MooreWitten} (See also \cite{ChoieLee}). This is the bracket operator used in Lemma \ref{EkQ} and, as noted, preserves modularity, but changes the weight from $\frac 32$ to $\frac32 + 2\ell$. A calculation shows that  $\left(\Theta_2(\tau)\Theta_3(\tau)\right)^{-2}$ and $\Theta_4(\tau)^{-1}$ are holomorphic away from infinity, which, combined with the fact that $\Theta_2(\tau)\Theta_3(\tau)\in \Z[[q^2]]$, shows that $\mathcal H_\ell(\tau)$ is in $M_0^*(\Gamma_0(8))$. Hence it suffices to show that $\mathcal G_\ell(q)-\mathcal H_\ell(q)$ is in $M_0^*(\Gamma_0(8))$ as well. From (\ref{KQ_as_E2*}), we see that

\begin{equation}
\begin{split}
\mathcal G _{\ell}(q)&-\mathcal H_\ell(q) = \\
&\sum_{j=0}^\ell  \begin{pmatrix}\ell\\j \end{pmatrix}\frac{\Gamma(\frac32) (-12)^j E_2(8\tau)^{\ell-j}\Theta_4(\tau)}{\left(\Theta_2(\tau)\Theta_3(\tau)\right)^{2\ell+2}\Gamma\left(\frac 32+j\right)8^j} \left[ \frac{(-4)^j 8 K_{2j}(8\tau)}{\eta^3(8\tau)} - \left(q\frac{d}{dq}\right)^j \frac{8K_{0}(8\tau)}{\eta^3(8\tau)} \right].
\end{split}
\end{equation}
Using Theorem \ref{K_in_Mu}, this can be written as

\begin{equation}\label{G-H}
\begin{split}
\mathcal G _{\ell}(q)-&\mathcal H_\ell(q)\\
=&\frac{\Theta_4(\tau)}{\left(\Theta_2(\tau)\Theta_3(\tau)\right)^{2\ell+2}} \sum_{j=0}^\ell  \begin{pmatrix}\ell\\j \end{pmatrix}\frac{\Gamma(\frac32) (-12)^j E_2(8\tau)^{\ell-j}}{\Gamma\left(\frac 32+j\right)8^{j-1}}\\
 &\times \left.\left[ (-1)^j D_u^{2j+1} D_v -D_\tau^{j}D_u D_v\right] \frac{\rho^{-1}q^{-1}\mu(2u+2v+4\tau,2v;8\tau)\theta(2v;4\tau)}{2\eta^3(8\tau)}  \right|_{u=v=0}.
\end{split}
\end{equation}

Paying particular attention to the derivatives of the $\mu$-function above, we use the transformation laws for $\mu$ found in Lemma \ref{Mu_Transformations}, and observe that the Mordel integrals that arise as obstructions to the modular transformation of  (\ref{G-H}) cancel directly. Therefore an argument similar to the proof that the bracket operator preserves modularity suffices to show that $\mathcal G _{\ell}(q)-\mathcal H_\ell(q)$ is modular with respect to $\Gamma_0(8)$. Some simple accounting shows that $\mathcal G _{\ell}(q)-\mathcal H_\ell(q)$ is supported on even exponents of $q$, and hence $\mathcal G _{\ell}(q)-\mathcal H_\ell(q)$ is in $M_0^*(\Gamma_0(8))$. This completes the proof.

\subsection{Examples}

In the table below, we give the polynomial $P_n(x)$ such that the  expression in the statement of Theorem \ref{Criterion} can be written as
\begin{eqnarray}
\left(q\frac{d}{dq}Z(q^{1/8})\right) \left(\frac{Z(q^{1/8})}{2}\right)^mP_n(Z(q^{1/8})).
\end{eqnarray}

\begin{center}

\begin{tabular}{|l|l|}
\hline \ & \ \\
$n$& $P_n(x)$\\ 
\hline
0&$\frac{1}{32}x$\\ \ & \ \\
1&$-1/2$\\ \ & \ \\
2&$\frac{13}{16}x$\\ \ & \ \\
3&$-\frac{11}{16}x^2-87$\\ \ &  \ \\
4&$\frac{13}{16}x^3 + \frac{4175}{8}x$\\ \ & \ \\
5&$-\frac{11}{16}x^4-\frac{9607}{4}x^2-80662$\\ \ & \ \\
6&$\frac{13}{16}x^5+\frac{80153}{8}x^3 + \frac{5958039}{4}$\\ \ & \ \\
\hline
\end{tabular}

\end{center}

\end{document}